\documentclass[microtype]{gtpart}     
\usepackage{graphicx}  

\title{$\ \ $ Structure and realizability for rational maps}

\author{Zhiqiang Wei}
\address{School of Mathematics and Statistics, Henan University, Kaifeng, China \newline
Center for Applied Mathematics of Henan Province, Henan University, Zhengzhou, China}
\email{weizhiqiang15@mails.ucas.edu.cn  ~~or~~10100123@vip.henu.edu.cn}


\arxivreference{}

\newtheorem{theorem}{Theorem}[section]    
\newtheorem{lem}[theorem]{Lemma}          
\theoremstyle{definition}
\newtheorem{defn}[theorem]{Definition}    
\newtheorem{remark}{Remark}             
\newtheorem{conjecture}{Conjecture}[section]
\newtheorem{corollary}{Corollary}[section]
\newtheorem{prop}{Proposition}[section]
\newtheorem{ex}{Example}[section]

\begin{document}

\begin{abstract}    
We establish a structure theorem for rational maps $f:\overline{\mathbb{C}}\to\overline{\mathbb{C}}$: the pullback metric
$f^{*}{\rm d}s_{0}^{2}$ of the standard metric ${\rm d}s_{0}^{2}$ admits a canonical decomposition into finitely many footballs---Riemann
spheres with two antipodal conical singularities of equal angle---by cutting
along a finite set of geodesics. This geometric decomposition provides a new
framework for the Hurwitz existence problem. As an application, we prove that a collection $\mathcal{D}$ of $k$ nontrivial partitions of a positive integer $d$ satisfying the Riemann--Hurwitz condition is realizable  as the branch datum of a rational map whenever $k>l+1$, where $l$ is the minimum partition length. This unifies the classical
results of Thom ($l = 1$), Pakovich ($l = 2$) and Bara\'{n}ski ($k\geq d$), and confirms a conjecture of
Zheng in an important special case.
\end{abstract}

\maketitle


\section{Introduction}

The problem of determining whether a given collection of nontrivial partitions of a positive integer can be realized by a branched covering between Riemann surfaces has a rich history dating back to the seminal work of Hurwitz~\cite{Hur91}. This problem, now known as the \emph{Hurwitz existence problem}, lies at the intersection of complex analysis, topology, and algebraic geometry, with profound connections to dessins d'enfants, Teichm\"uller theory, and the study of spherical conic metrics.

Let us recall the classical formulation. Let $M$ and $N$ be compact Riemann surfaces with Euler characteristic $\chi(M)$ and $\chi(N)$ respectively, and let $f: M \to N$ be a branched covering of degree $d$. Denote the set of branch points of $f$ by $B_f$, then for each  point $x \in B_f$, the local behavior of $f$ is described by a partition $\pi(x) = [\alpha_1,\ldots,\alpha_r]$ of $d$, recording the local degrees at the preimages of $x$.  The collection $\mathcal{D} = \{\pi(x) : x \in B_f\}$ is called the \emph{branch datum} of $f$. The Riemann--Hurwitz formula imposes a natural necessary condition:
\[
\nu(\mathcal{D}) = \sum_{\pi \in \mathcal{D}} (d - |\pi|) = d \cdot \chi(N) - \chi(M),
\]
where $|\pi|$ denotes the length of the partition $\pi$. A collection $\mathcal{D}$ satisfying this condition is called a \emph{candidate branch datum}. The Hurwitz existence problem asks: given $M$, $N$, and a candidate branch datum $\mathcal{D}$, does there exist a branched covering $f: M \to N$ with $\mathcal{D}$ as its  branch datum?

When the target surface $N$ is not the $2$-sphere, the natural necessary condition is also sufficient~\cite{EKS84, EZ78, Hus62}. However, the case $N = S^2$ is substantially more subtle. Despite intensive study over the past century, a complete characterization of realizable candidate branch data remains elusive; data that satisfy the Riemann--Hurwitz formula but cannot be realized are termed \emph{exceptional}, and their classification is an open problem~\cite{FP24, WWX25-1, Zhe06, Zhu19}.

Significant progress has been made in recent decades. Thom~\cite{Th65} and Edmonds--Kulkarni--Stong~\cite{EKS84} proved that if $\mathcal{D}$ contains a partition of length one, it is always realizable. The case where $\mathcal{D}$ contains a partition of length two was subsequently resolved through the combined efforts of Edmonds--Kulkarni--Stong~\cite{EKS84}, Pervova--Petronio~\cite{PP06}, Pakovich~\cite{Pak09}, and Baroni-Petronio~\cite{BP24}, with the latter providing a complete solution. The case where the number of partitions of $\mathcal{D}$ $\geq d$ was solved by Bara\'{n}ski~\cite{Bar01}. These results suggest a general principle: the realizability of a candidate branch datum is intimately related to the minimal length among its partitions and the number of branch points.

Building on computational evidence, Zheng~\cite{Zhe06} proposed a conjecture that captures this principle:

\begin{conjecture}\cite{Zhe06}\label{conj}
Let $d = d' d''$, where $d'$ is the minimal prime factor of $d$ and $d'' > 2$ (i.e., $d$ is neither prime nor equal to $4$). Then a candidate branch datum $\mathcal{D}$ with $k$ partitions is realizable provided that $k > d'' + 1$. Moreover, the bound $k > d'' + 1$ is sharp.
\end{conjecture}

For more details about the Hurwitz existence problem, we refer to \cite{Bar01,Bo82,Ger87,Hus62,KZ96,PP09,PP12,PP06,PP08,PP07,P20,SX20}. In particular, \cite{PP07, P20} provide comprehensive reviews of available results and techniques.

This paper addresses the Hurwitz existence problem from a novel geometric perspective. By identifying the $2$-sphere $S^2$ with the Riemann sphere $\overline{\mathbb{C}}$, Stoilow's theorem~\cite{RP19,S28} allows us to treat any orientation-preserving branched covering $f: \overline{\mathbb{C}} \to \overline{\mathbb{C}}$ as a rational map. Endowing the target sphere with the standard metric of constant curvature $1$, denoted by $\mathrm{d}s_0^2$, we consider its pullback $f^*\mathrm{d}s_0^2$ on the source sphere. This pullback metric is a smooth Riemannian metric everywhere except at the critical points of $f$, where it exhibits conical singularities. It is, by construction, a metric of constant curvature $1$ on the complement of its singularities. Metrics of this type, known as \emph{spherical conic metrics}, have been the subject of intense investigation in recent years (e.g. \cite{BD13,DFA11,Chen15,Er21,Er23,Hei62,LT92,MZ20,MZ22,McO88,MD16,MD19,EP05,HP98,Tr89,Tr91,WZ00,WWX26}), and our work leverages this rich theory to gain new insights into the classical problem of realizing branched coverings.

Our first main result reveals that this metric admits a remarkably simple and canonical decomposition:

\begin{theorem}[Structure theorem for rational maps]\label{Thm1}
Let $f: \overline{\mathbb{C}} \to \overline{\mathbb{C}}$ be a rational map with degree $\geq 2$, and let $\mathrm{d}s_0^2$ be the constant curvature $1$ metric on the target. Then the pullback metric $f^*\mathrm{d}s_0^2$ is a constant curvature $1$ metric with finitely many conical singularities on the source. Moreover, the space $(\overline{\mathbb{C}}, f^*\mathrm{d}s_0^2)$ admits a canonical decomposition: by cutting along a finite set of geodesics connecting the poles, zeros and the critical points of $f$, it can be partitioned into finitely many pieces, each isometric to a \emph{football}---a $2$-sphere with two conical singularities at antipodal points of equal angle.
\end{theorem}

This structure theorem is motivated by the recent work of Wu, Xu, and the author~\cite{WWX26} on reducible spherical metrics, and provides a powerful geometric tool for constructing rational maps with prescribed branch data. By assembling footballs along geodesic boundaries, rational maps can be built in a combinatorial fashion, thereby reducing the Hurwitz existence problem to a gluing construction.

As an application of this geometric framework, we prove the following existence result:

\begin{theorem}\label{Thm2}
Let $k \geq 3$ and $d \geq 3$ be integers. Let $\pi_1, \ldots, \pi_k$ be $k$ nontrivial partitions of $d$ such that
\[
\sum_{i=1}^k (d - |\pi_i|) = 2d - 2,
\]
and assume $|\pi_i| \geq |\pi_k|$ for all $i$. If $k > |\pi_k| + 1$, then the collection $\mathcal{D} = \{\pi_1, \ldots, \pi_k\}$ is realizable.
\end{theorem}

This theorem unifies and extends the classical results of Thom (the case $|\pi_k| = 1$), Pakovich (the case $|\pi_k| = 2$)  and Bara\'{n}ski (the case $k\geq d$). Moreover, it allows us to confirm Zheng's conjecture in an important special case:

\begin{theorem}\label{Thm3}
Suppose $d = d' d''$, where $d'$ is the minimal prime factor of $d$ and $d'' > 2$. Then any candidate branch datum $\mathcal{D}$ that contains a partition of length $d''$, and has $k > d'' + 1$ partitions is realizable. The condition $k > d'' + 1$ is sharp, as demonstrated by an explicit non-realizable example with $k = d'' + 1$.
\end{theorem}

\begin{proof}
Since $\mathcal{D}$ has $k$ elements and $k > d'' + 1$, by Theorem \ref{Thm2}, $\mathcal{D}$ is realizable.

Since
\[
d'' + 1 > d'' = \frac{d}{\gcd(d', \ldots, d')},
\]
it follows from a result in \cite{SX20}  that the collection
\[
\{[d', \ldots, d'], [d', \ldots, d'], [1, \ldots, 1, d'' + 1], \underbrace{[1, \ldots, 1, 2], \ldots, [1, \ldots, 1, 2]}_{d''-2}\}
\]
(with $k = d'' + 1$ and $\nu(\mathcal{D}) = 2d - 2$) is non-realizable. This completes the proof.
\end{proof}

The paper is organized as follows. Section~2 reviews essential background on branched coverings, conical singularities, and football metrics. Section~3 develops the geometric structure theory and proves Theorem~\ref{Thm1}. This section also presents elementary constructions of rational maps with branch datum $\{[2],[2]\}$ via football gluing. Moreover, it provides some explicit examples to illustrate the geometric ideas behind the inductive proof. Section 4 establishes combinatorial properties of candidate branch data and delivers the
proof of Theorem \ref{Thm2}. Section 5 lists some open problems for future investigation.

\section{Preliminary}
In this section, we introduce the notions of a \emph{branched covering}, a \emph{conical singularity} and a \emph{football metric}---or simply a \emph{football}---which is the simplest type of constant Gaussian curvature $1$ metrics with conical singularities and serves as a basic building block in our construction.

\subsection{Branched covering and Hurwitz existence problem}

Let \( M \) and \( N \) be compact Riemann surfaces with Euler characteristics \( \chi(M) \) and \( \chi(N) \). A smooth map \( f: M \to N \) is termed a \emph{branched covering} of degree \( d \) provided that for every \( x \in N \), there exists a partition \( \pi(x) = [\alpha_1, \ldots, \alpha_r] \) of \( d \)---representing the local multiplicities---such that \( f \) is locally modeled by
\[
(j, z) \mapsto z^{\alpha_j}, \quad z \in \mathbb{D}.
\]
The \emph{branch set} \( B_f \subset N \) is the finite set of points \( x \) for which \( \pi(x) \neq [1, \ldots, 1] \), and the \emph{branch datum} is the collection \( \mathcal{D} = \{\pi(x) : x \in B_f\} \).

A fundamental constraint is given by the Riemann--Hurwitz formula:
\begin{equation}\label{RHF}
\nu(\mathcal{D}) = d \cdot \chi(N) - \chi(M).
\end{equation}
Here, the total branching \( \nu(\mathcal{D}) \) is defined in terms of the local degrees. More precisely, if \( B_f = \{x_1, \ldots, x_k\} \) and \( \mathcal{D} = \{[\alpha_{i1}, \ldots, \alpha_{ir_i}]\}_{i=1}^k \), then near each  point \( y_{ij} \) in the preimage of \( x_i \), the map is locally \( z \mapsto z^{\alpha_{ij}} \). Consequently,
\[
\nu(\mathcal{D}) = \sum_{i=1}^{k} \sum_{j=1}^{r_i} (\alpha_{ij} - 1).
\]
Observing that \( \sum\limits_{j=1}^{r_i} \alpha_{ij} = d \), we obtain the equivalent form:
\begin{equation}\label{RHF-1}
\sum_{i=1}^{k} (d - r_i) = d \cdot \chi(N) - \chi(M).
\end{equation}

A pair \((d, \mathcal{D})\) satisfying \eqref{RHF} is called a \emph{candidate branch datum}; by a slight abuse of terminology, we also refer to \(\mathcal{D}\) itself as a candidate branch datum when \(d\) is understood. The \emph{Hurwitz existence problem} asks for the existence of a branched covering \(f: M \to N\) of degree \(d\) with branch datum \(\mathcal{D}\), for given compact Riemann surfaces \(M\) and \(N\) and a candidate branch datum \((d, \mathcal{D})\).

\subsection{Conical singularities and footballs}

\begin{defn}[\cite{WZ00}]
 Let $\mathrm{d}s^2 = e^{2\psi}|\mathrm{d}z|^2$ be a conformal metric on the punctured disk $\mathbb{D} \setminus \{0\}$. The singular point $z = 0$ is called a \emph{conical point} with singular angle $2\pi\alpha$ (where $0 < \alpha \neq 1$) if and only if $\psi$ admits a local representation of the form
\[
\psi(z) = (\alpha - 1)\ln |z| + \rho(z),
\]
where $\rho(z)$ is a smooth function on $\mathbb{D}=\{z\in\mathbb{C}| |z|<1\}$. We say that $\mathrm{d}s^2$ is a conical metric on $\mathbb{D}$.
\end{defn}

A conical metric of constant Gaussian curvature $K = -1$, $0$, or $1$ is called \emph{hyperbolic}, \emph{flat}, or \emph{spherical}, respectively. A fundamental problem in complex analysis and differential geometry is to determine the existence and uniqueness of such a metric associated with a given real divisor
\[
D = \sum_{l=1}^{N} (\alpha_l - 1) P_l,
\]
where $\alpha_l > 0$ and $\alpha_l \neq 1$ for all $l$, on a compact Riemann surface $M$. This problem extends the classical uniformization theorem to surfaces with conical singularities and has been extensively studied. We do not go into details on this topic in the present paper; readers who are interested may find further information in the references cited herein.

\begin{defn}(\cite{Tr89,WWX26})
We call a spherical conical metric on $S^{2}\cong\mathbb{C}\cup \{\infty\}$ a football metric if it has two equal cone angles at antipodal points. This metric can be expressed explicitly as ${\rm d}s^{2}=\frac{4\alpha^{2}|z|^{2(\alpha-1)}}{(1+|z|^{2\alpha})^{2}}|{\rm d}z|^{2}$, and is denoted by $S^{2}_{\{\alpha,\alpha\}}$ in this paper, where $\alpha>0$ is the cone angle parameter.
\end{defn}

\begin{figure}[htbp]
\centering
\includegraphics[width=8cm]{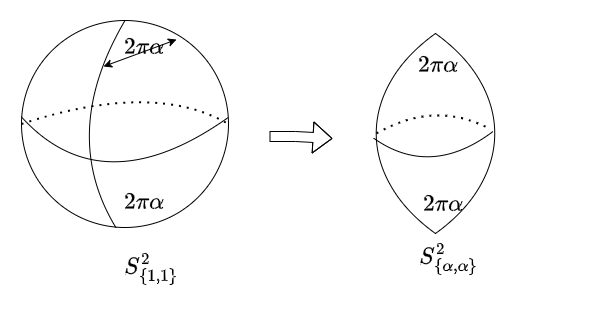}
\caption{An example of a football metric.}
\end{figure}

The geodesic distance between the two singularities of a football metric is $\pi$. In the present paper, for convenience, we also regard the standard smooth metric on $S^{2}$ as a football metric and denote it by $S^{2}_{\{1,1\}}$. Furthermore, a $2$-sphere equipped with a football metric is referred to as a \emph{football}.

For example, on the standard football $S^{2}=S^{2}_{\{1,1\}}$, by taking a bigon with an angle $2\pi\alpha(0<\alpha<1)$ (as shown in Figure 1), we can construct an American football $S^{2}_{\{\alpha,\alpha\}}$.

\section{The structure of rational maps and proof of Theorem \ref{Thm1}}
\subsection{Geometric properties of rational maps}

Let $f: \overline{\mathbb{C}} \to \overline{\mathbb{C}}$ be a rational map of degree $d$ with branch data $\{\pi_1, \ldots, \pi_k\}$, where each
\[
\pi_i = [a_{i1}, \ldots, a_{ir_i}], \quad i = 1, \ldots, k,
\]
is a nontrivial partition of $d$. After a suitable M\"{o}bius transformation, $f$ can be written as
\begin{equation}\label{Exp_R}
f(z) = \frac{\prod\limits_{j=1}^{r_1} (z - z_j)^{a_{1j}}}{\prod\limits_{l=1}^{d} (z - w_l)},
\end{equation}
where $z_1, \ldots, z_{r_1}, w_1, \ldots, w_d$ are distinct complex numbers.

Note that the proof of Theorem \ref{Thm1} depends on the expression of
$f$. However, different rational maps may share the same branch datum. Therefore, when applying Theorem~\ref{Thm1} to establish the realizability of a rational map with prescribed branch datum, the key is to choose an appropriate expression of $f$. For instance, in the proof of Theorem ~\ref{Thm2} and all examples given in section 3.4 we adopt the expression given in \eqref{Exp_R}.

To uncover the hidden geometric structure, we construct two key tools: a real-valued function $\Phi$ that measures the ``logarithmic scale'' of $f$, and a related Killing vector field $\vec{V}$. The integral curves of $\vec{V}$ and the gradient lines of $\Phi$ together form a coordinate grid that naturally demarcates the boundaries of the football pieces in the decomposition. \par

Define a real function $\Phi: \overline{\mathbb{C}} \to \mathbb{R}$ by
\[
\Phi(p) = \frac{4|f(p)|^2}{1 + |f(p)|^2}, \quad p \in \overline{\mathbb{C}}.
\]
Clearly $\Phi$ is smooth on $\overline{\mathbb{C}}$ and the range of $\Phi$  is $[0, 4]$. Moreover:
\begin{itemize}
  \item The zeros of $f$ are minima of $\Phi$,
  \item The poles of $f$ are maxima of $\Phi$,
  \item The critical points of $f$ (excluding zeros) are saddle points of $\Phi$.
\end{itemize}

Now define a meromorphic $1$-form by
\[
\omega = \frac{\mathrm{d}f}{f}.
\]
Then $\omega$ has only simple poles, with residues $a_{1j}$ at $z_j$ for $j = 1, \ldots, r_1$, and $-1$ at each $w_l$ for $l = 1, \ldots, d$. In particular, each point corresponding to an entry $a_{ij} > 1$ in a partition $\pi_i$ with $i = 2, \ldots, k$ is a zero of $\omega$ of order $a_{ij} - 1$, and also a saddle point of $\Phi$.

The following theorem summarizes the geometric structure of $f: \overline{\mathbb{C}} \to \overline{\mathbb{C}}$:

\begin{theorem}\label{Thm3.1}
Let $\mathrm{d}s^{2}_{0}$ be the standard metric on $\overline{\mathbb{C}}$. Then $\mathrm{d}s^{2} = f^*\mathrm{d}s^{2}_{0}$ is a constant Gaussian curvature $1$ metric on $\overline{\mathbb{C}}$ with finitely many conical singularities. These occur at the critical points of $f$, and the conical angle at a critical point of multiplicity $a_{ij} > 1$ is $2\pi a_{ij}$. Moreover, $\Phi$ and $\omega$ satisfy
\begin{equation}\label{E-1}
\frac{4\,\mathrm{d}\Phi}{\Phi(4 - \Phi)} = \omega + \overline{\omega},
\end{equation}
and the metric $\mathrm{d}s^{2}$ is given by
\[
\mathrm{d}s^{2} = \frac{\Phi(4 - \Phi)}{4} \, \omega \, \overline{\omega}.
\]
\end{theorem}

\begin{proof}
These results follow by direct computation.
\end{proof}

Finally, define a complex vector field $X$ by the condition
\[
\omega(X) = \sqrt{-1}.
\]
Then $X$ is a meromorphic vector field on $\overline{\mathbb{C}}$. Its real part,
\[
\vec{V} = \frac{1}{2}(X + \overline{X}) = \operatorname{Re}(X),
\]
is a real vector field on the sphere.

The following properties can be verified directly:

\begin{prop}\label{Pro-2-2}
Let $\operatorname{Sing}(\vec{V})$ denote the set of singular points of $\vec{V}$. Then
\[
\operatorname{Sing}(\vec{V}) = \{\text{poles and zeros of } \omega\} = \{\text{poles, zeros, and critical points of } f\}.
\]
Moreover, $\vec{V}$ is a Killing vector field on $\overline{\mathbb{C}}\setminus \operatorname{Sing}(\vec{V})$ and satisfies $\vec{V} \perp \nabla \Phi$, where $\nabla \Phi$ denotes the real gradient of $\Phi$.
\end{prop}

Let $\Omega_{p}$ denote the set of all integral curves of $\vec{V}$ that intersect at a point $p$, with cardinality $|\Omega_{p}|$. For any $p \notin \operatorname{Sing}(\vec{V})$, there exists a unique integral curve of $\vec{V}$ passing through $p$, so $\Omega_{p}$ is well-defined and consists of exactly two elements: one entering and one leaving $p$. However, at a singular point $p \in \operatorname{Sing}(\vec{V})$, the notion of an integral curve passing through $p$ becomes ambiguous. To resolve this, we say an integral curve $C$ belongs to $\Omega_{p}$ (for $p \in \operatorname{Sing}(\vec{V})$) if and only if there exists a sequence of points in $C$ converging to $p$. Equivalently, if $C \notin \Omega_{p}$, then there exists a small disk $B$ centered at $p$ such that $C \cap B = \emptyset$.

\begin{prop}\label{Pro-2-3}
In a small neighborhood of any extremum of $\Phi$, every integral curve of $\vec{V}$ is a topological circle enclosing the extremum in its interior.
\end{prop}

\begin{proof}
Without loss of generality, assume that near an extremum of $\Phi$, we have $\omega = \frac{\alpha}{z} \mathrm{d}z$ with $\alpha=-1$ or $\alpha \in \mathbb{Z}^{+}$. Then $Y = \sqrt{-1} \frac{z}{\alpha} \frac{\partial}{\partial z}$. Writing $z = x + \sqrt{-1}y$, we obtain
\[
\vec{V} = \frac{1}{2\alpha} \left( x \frac{\partial}{\partial y} - y \frac{\partial}{\partial x} \right).
\]
It is clear that every integral curve of $\vec{V}$ is a topological circle enclosing the extremum.
\end{proof}

\begin{prop}\label{Pro-2-4}
The set $\operatorname{Sing}(\vec{V})$ decomposes into two disjoint subsets, $\operatorname{Sing}(\vec{V}) = S_1 \cup S_2$, where:
\begin{enumerate}
\item $S_1 = \{ p \in \operatorname{Sing}(\vec{V}) \mid \Omega_p = \emptyset \}$, and if $p \in S_1$, then $p$ is an extremum of $\Phi$.
\item $S_2 = \{ p \in \operatorname{Sing}(\vec{V}) \mid \Omega_p \neq \emptyset \}$, and if $p \in S_2$, then $p$ is a saddle point of $\Phi$ and the angle of $\mathrm{d}s^2$ at $p$ is $\pi \cdot |\Omega_p|$.
\end{enumerate}
\end{prop}

\begin{proof}
Case (1) follows from Proposition \ref{Pro-2-3}. We now prove Case (2). Without loss of generality, near a saddle point $p$ of $\Phi$, choosing a new complex coordinate $z$, we can assume $\omega = z^n \mathrm{d}z$ with $n \in \mathbb{Z}^{+}$, where $z$ is a local complex coordinate such that no other singularities of $\vec{V}$ lie in the neighborhood. Then $Y = \frac{\sqrt{-1}}{z^n} \frac{\partial}{\partial z}$. Writing $z = r e^{\sqrt{-1}\theta}$, we compute:
\begin{align*}
Y &= \frac{\sqrt{-1}}{r^n e^{\sqrt{-1}n\theta}} \left( \frac{\partial r}{\partial z} \frac{\partial}{\partial r} + \frac{\partial \theta}{\partial z} \frac{\partial}{\partial \theta} \right)
= \frac{\sqrt{-1}}{2 r^n e^{\sqrt{-1}(n+1)\theta}} \left( \frac{\partial}{\partial r} - \frac{\sqrt{-1}}{r} \frac{\partial}{\partial \theta} \right) \\
&= \frac{1}{2 r^n} \left[ \sin((n+1)\theta) \frac{\partial}{\partial r} + \frac{\cos((n+1)\theta)}{r} \frac{\partial}{\partial \theta} + \sqrt{-1}(\cdots) \right].
\end{align*}
Hence, the real vector field $\vec{V}$ is given by:
\[
\vec{V} = \frac{\sin((n+1)\theta)}{2 r^n} \frac{\partial}{\partial r} + \frac{\cos((n+1)\theta)}{2 r^{n+1}} \frac{\partial}{\partial \theta}.
\]

Suppose an integral curve of $\vec{V}$ near $p$ is parametrized as:
\[
\begin{cases}
r = r(t), \\
\theta = \theta(t),
\end{cases}
\quad t \in (-\varepsilon, \varepsilon) \setminus \{0\},\ \varepsilon > 0,
\]
with $\lim\limits_{t \to 0} r(t) = 0$. Then the system becomes:
\begin{equation}\label{P-2-E-1}
\begin{cases}
r' = \dfrac{\sin((n+1)\theta)}{2 r^n}, \\[10pt]
\theta' = \dfrac{\cos((n+1)\theta)}{2 r^{n+1}}.
\end{cases}
\end{equation}

If $\theta' \neq 0$, the solution of \eqref{P-2-E-1} satisfies:
\[
r^{-(n+1)} = \lambda |\cos((n+1)\theta)|, \quad \lambda > 0,
\]
i.e.,
\[
r = \frac{1}{\sqrt[n+1]{\lambda |\cos((n+1)\theta)|}} \geq \frac{1}{\sqrt[n+1]{\lambda}},
\]
so such a curve does not reach $p$.

If $\theta' = 0$, the solutions are:
\[
\begin{cases}
r^{n+1} = -\dfrac{n+1}{2} t, \\
\theta = \dfrac{1}{n+1} \left[ \dfrac{\pi}{2} + (2k+1)\pi \right], \quad k = 0, \ldots, n,
\end{cases}
\quad t < 0,
\]
and
\[
\begin{cases}
r^{n+1} = \dfrac{n+1}{2} t, \\
\theta = \dfrac{1}{n+1} \left( \dfrac{\pi}{2} + 2k\pi \right), \quad k = 0, \ldots, n,
\end{cases}
\quad t > 0.
\]
This completes the proof of Case (2).
\end{proof}

\begin{prop}\label{Pro-2-5}
At a saddle point of $\Phi$, the angle between two adjacent integral curves of the gradient field $\nabla \Phi$ is $\pi$.
\end{prop}

\begin{proof}
By the proof of Proposition \ref{Pro-2-4}, the angle between two adjacent integral curves of $\vec{V}$ at a saddle point is $\pi$. Since $\nabla \Phi \perp \vec{V}$ by Proposition \ref{Pro-2-2}, the same holds for the integral curves of $\nabla \Phi$.
\end{proof}

According to Proposition \ref{Pro-2-3}, if $p$ is a local minimum point of $\Phi$, the integral curves of $\vec{V}$ form topologically concentric circles in a small neighborhood of $p$. Moreover, the integral curves of $\nabla \Phi$ are perpendicular to these circles. Let $c(t)~(t\in[0,T])$ be an integral curve of $\vec{V}$. For each $t$, there exists a unique integral curve $C_{t}$ (i.e., geodesic) of $\nabla \Phi$ originating from $p$ and intersecting the point $c(t)$. One observes that $C_{t}$ must reach either a saddle point or a local maximum point of $\Phi$ (refer to Figure 2).

 \begin{figure}[htbp]
\centering
\includegraphics[width=6.5cm]{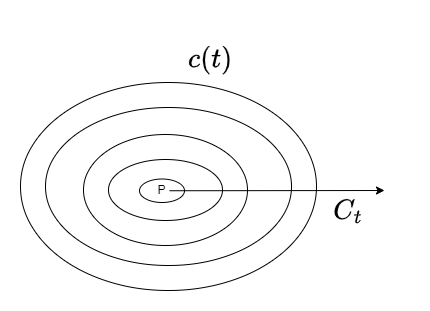}
\caption{Integral curves}
\end{figure}

\begin{prop}\label{Pro-2-6}
Assuming that $C_{t}$ reaches the local maximum point $q$ directly without passing through any saddle point of $\Phi$, the arc length of $C_{t}$ is equal to $\pi$.
\end{prop}
\begin{proof}~
Suppose $r:[0,l]\rightarrow \overline{\mathbb{C}},s\mapsto r(s)$ is the parameter equation of curve $C_{t}$, where $s$ is the arc length parameter. Then
$$l=\int_{0}^{l}{\rm d}s=\int_{0}^{l}{\rm d}r=\int_{0}^{4}\frac{{\rm d}r}{{\rm d}\Phi}{\rm d}\Phi= \int_{0}^{4}\frac{1}{\sqrt{\Phi(4-\Phi)}}{\rm d}\Phi=\pi.$$
\end{proof}

From Proposition  \ref{Pro-2-6}, we obtain the following corollary.

\begin{corollary}\label{Co-2-1}
For any $t_{1}, t_{2} \in [0, T]$, if both $C_{t_{1}}$ and $C_{t_{2}}$ reach local maxima at points $q_{1}$ and $q_{2}$ respectively without traversing any saddle point of $\Phi$, then the geodesic distance between  $p$ and $q_{1}$ is equal to the geodesic distance between $p$ and $q_{2}$, both being $\pi$.
\end{corollary}

From Proposition  \ref{Pro-2-6}, we derive the following property.

\begin{prop}
The function $\Phi(r)$ monotonically increases along a geodesic connecting a minimum point to a maximum point of $\Phi$.
\end{prop}

\begin{prop}\label{Pro-2-7}
 Let $t_{0}\in (0,T)$ be fixed. Suppose that $C_{t_{0}}$ directly reaches a maximum point $q$ of $\Phi$. Then, there exists $\varepsilon>0$ such that for all $t\in(t_{0}-\varepsilon,t_{0}+\varepsilon)$, $C_{t}$ also reaches the same maximum point $q$ without passing through any saddle point of $\Phi$.
\end{prop}
\begin{proof}~
Since the number of saddle points of $\Phi$ is finite, there exists a small neighborhood $(t_{0}-\varepsilon, t_{0}+\varepsilon)$ with $\varepsilon>0$ centered at $t_{0}$ where each curve $C_{t}$ for $t\in(t_{0}-\varepsilon, t_{0}+\varepsilon)$ does not intersect any saddle point of $\Phi$. The endpoints of $C_{t}$ vary continuously with $t$, and the maximum points of $\Phi$ are limited in number. Thus, there exists $\varepsilon>0$ such that for all $t\in(t_{0}-\varepsilon, t_{0}+\varepsilon)$, $C_{t}$ reaches the same maximum point.
\end{proof}

\subsection{Proof of Theorem \ref{Thm1}}
The proof of Theorem \ref{Thm1} follows from Theorem \ref{Thm3.1},  and Propositions \ref{Pro-2-4} and \ref{Pro-2-7}.

\subsection{Constructions for rational maps with branch datum $\{[2],[2]\}$}
By Theorem~\ref{Thm1}, for any rational map $f : \overline{\mathbb{C}} \to \overline{\mathbb{C}}$, the space $(\overline{\mathbb{C}}, f^*\mathrm{d}s^2_0)$ admits a decomposition into pieces, each isometric to a football. This enables the construction of rational maps by gluing suitable footballs. In particular, one may assemble such maps from standard footballs $S^2_{\{1,1\}}$. In this section, we present several elementary constructions of rational maps with branch datum $\{[2],[2]\}$ via this approach. We remark that Method~4 demonstrates that rational maps can be constructed by gluing footballs along arbitrary suitable geodesics.

Let $\Phi(z) = \frac{4|z|^2}{1 + |z|^2}$ be the function defined on a standard football $S^2_{\{1,1\}}\cong \mathbb{C}\cup\{\infty\}$.

\noindent\textbf{Method 1.}
Take two standard footballs and cut each along a geodesic connecting the maximum and minimum of $\Phi$, as shown in Figure~\ref{fig:case2}. The resulting surface corresponds to a rational map $f$ with two branch points and branch datum $\{[2],[2]\}$. Since one branch point is a minimum of $\Phi$ and the other is neither a maximum nor a minimum, they correspond to a zero and a non-extremal point of $f$, respectively. Hence, $f$ takes the form
\[
f(z) = \frac{(z - a)^2}{(z - b_1)(z - b_2)} = \frac{z - a}{z - b_1} \cdot \frac{z - a}{z - b_2},
\]
where $a, b_1, b_2$ are distinct complex numbers. Note that the gluing identifies points with equal function values along the incision.

\begin{figure}[htbp]
\centering
\includegraphics[width=7cm]{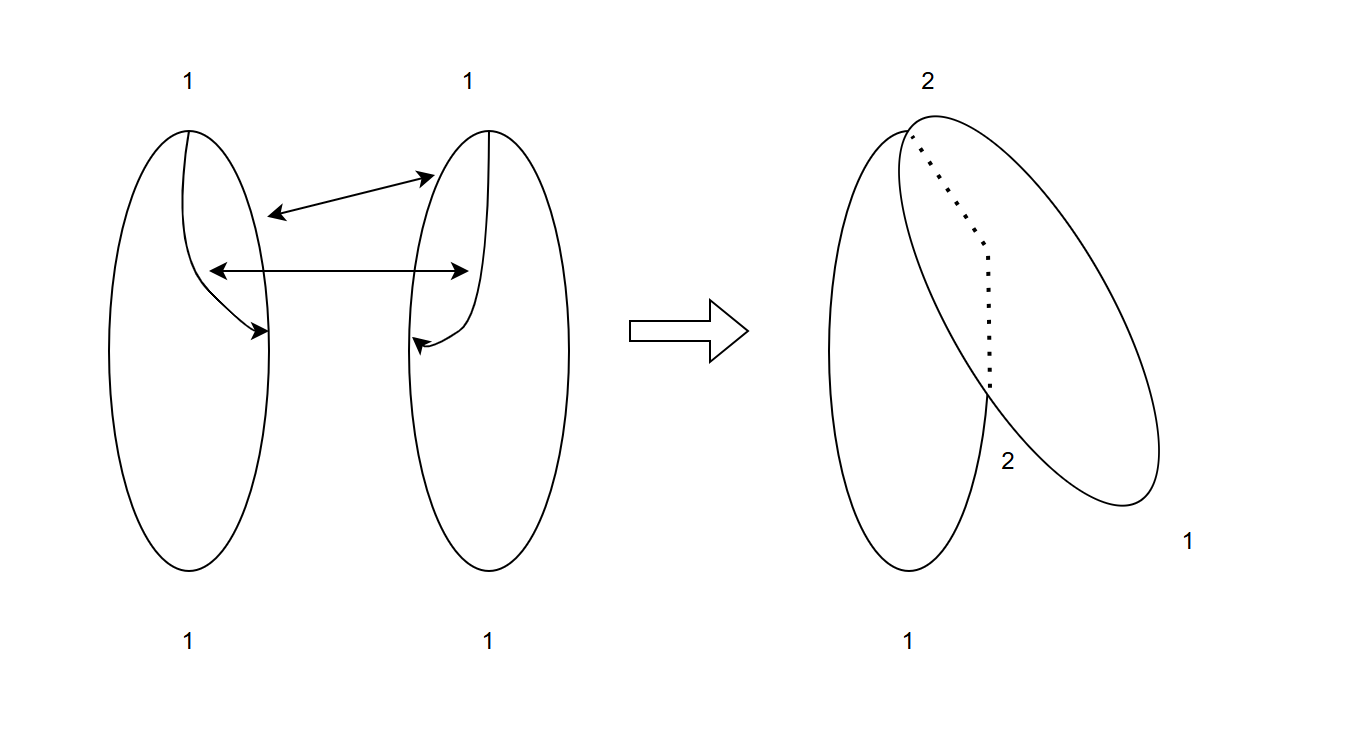}
\caption{Construction of $f(z) = \frac{(z - a)^2}{(z - b_1)(z - b_2)}$.}
\label{fig:case2}
\end{figure}

\noindent\textbf{Method 2.}
Take two standard footballs and cut each along a geodesic connecting the maximum and minimum of $\Phi$, as shown in Figure~\ref{fig:Method2}. The resulting surface corresponds to a rational map $f$ with two branch points and branch datum $\{[2],[2]\}$. Since neither branch point is an extremum of $\Phi$, they correspond to non-extremal points of $f$. Therefore, $f$ takes the form
\[
f(z) = \frac{(z - a_1)(z - a_2)}{(z - b_1)(z - b_2)} = \frac{z - a_1}{z - b_1} \cdot \frac{z - a_2}{z - b_2},
\]
where $a_1, a_2, b_1, b_2$ are distinct complex numbers. Note that the gluing identifies points with equal function values along the incision.

\begin{figure}[htbp]
\centering
\includegraphics[width=7cm]{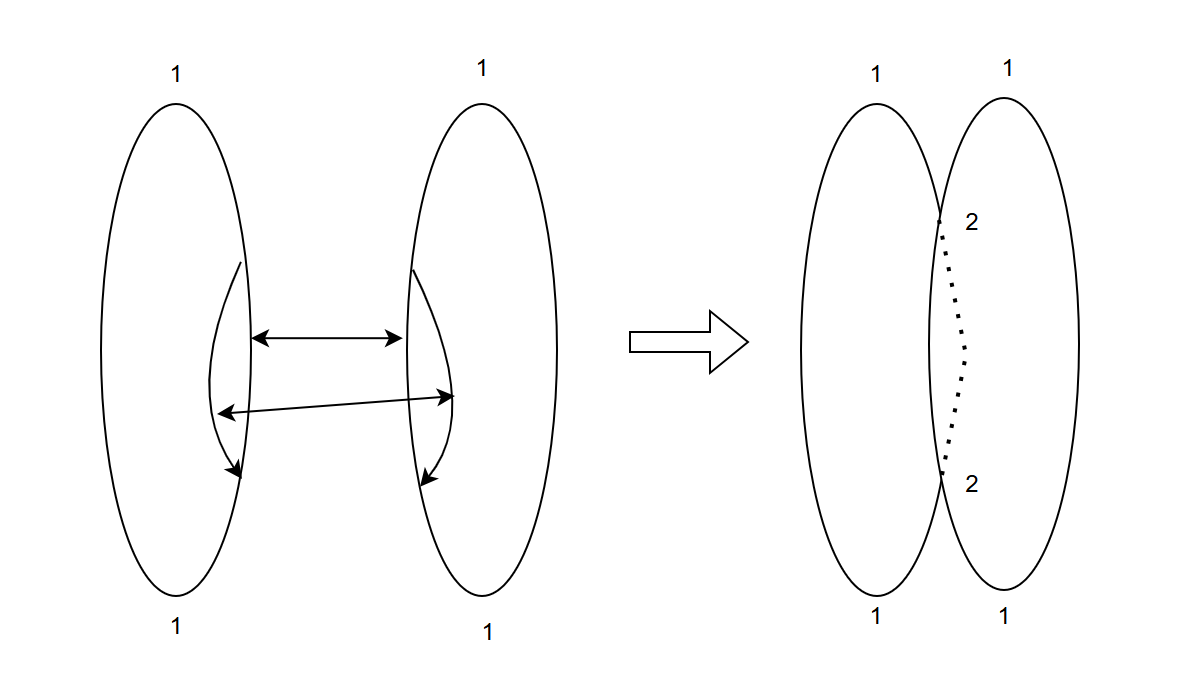}
\caption{Construction of $f(z) = \frac{(z - a_1)(z - a_2)}{(z - b_1)(z - b_2)}$.}
\label{fig:Method2}
\end{figure}

\noindent\textbf{Method 3.}
Take two standard footballs and cut each along a geodesic connecting the maximum and minimum of $\Phi$, as shown in Figure~\ref{fig:Method3}. Glue the two footballs along these cuts. The resulting surface corresponds to a rational map $f$ with two branch points and branch datum $\{[2],[2]\}$. Since one branch point is a maximum and the other a minimum of $\Phi$, they correspond to a pole and a zero of $f$, respectively. Hence, $f$ takes the form
\[
f(z) = \frac{(z - a)^2}{(z - b)^2} = \frac{z - a}{z - b} \cdot \frac{z - a}{z - b},
\]
where $a$ and $b$ are distinct complex numbers. In fact, we construct a football $S^{2}_{\{2,2\}}$. Note that the gluing identifies points with equal function values along the incision.

\begin{figure}[htbp]
\centering
\includegraphics[width=6cm]{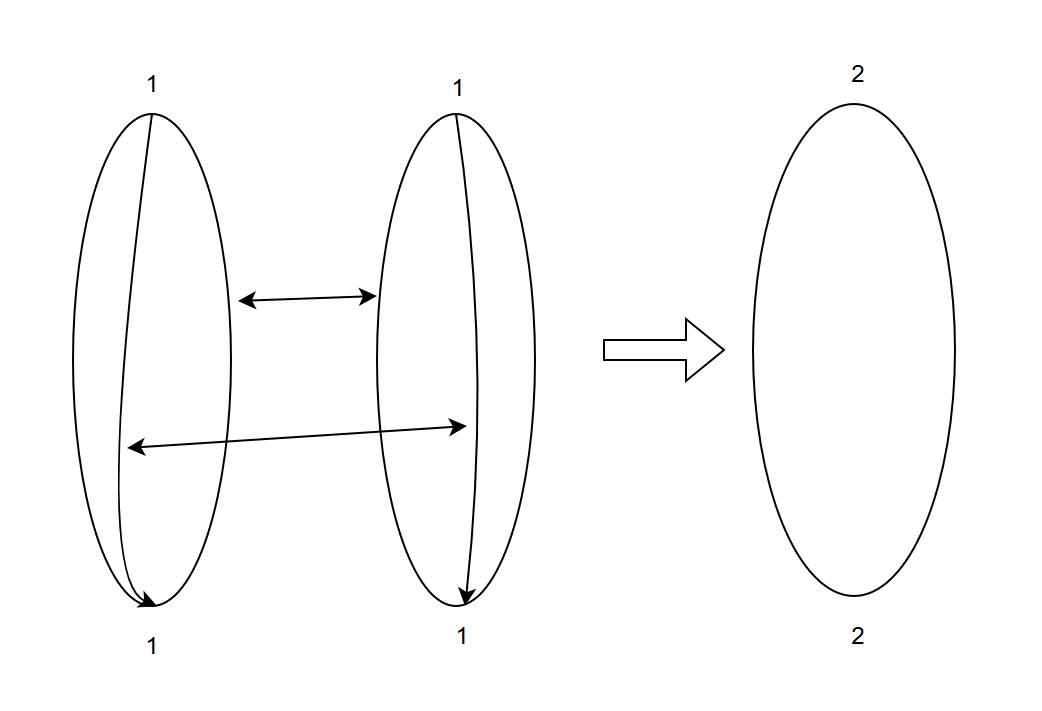}
\caption{Construction of $f(z) = \frac{(z - a)^2}{(z - b)^2}$.}
\label{fig:Method3}
\end{figure}

\noindent\textbf{Method 4.}
Begin with two standard footballs $S^2_{\{1,1\}}$, from which we obtain four pieces as illustrated in Figure~\ref{fig:Method4}.  Gluing the four pieces along these cuts, we can construct a  surface that corresponds to a rational map $f$ with two branch points and branch datum $\{[2],[2]\}$. Clearly, $f$ takes the form
\[
f(z) = \frac{(z - a_1)(z - a_2)}{(z - b_1)(z - b_2)} = \frac{z - a_1}{z - b_1} \cdot \frac{z - a_2}{z - b_2},
\]
where $a_1, a_2, b_1, b_2$ are distinct complex numbers.  Note that the gluing identifies points with equal function values along the incision.

\begin{figure}[htbp]
\centering
\includegraphics[width=8cm]{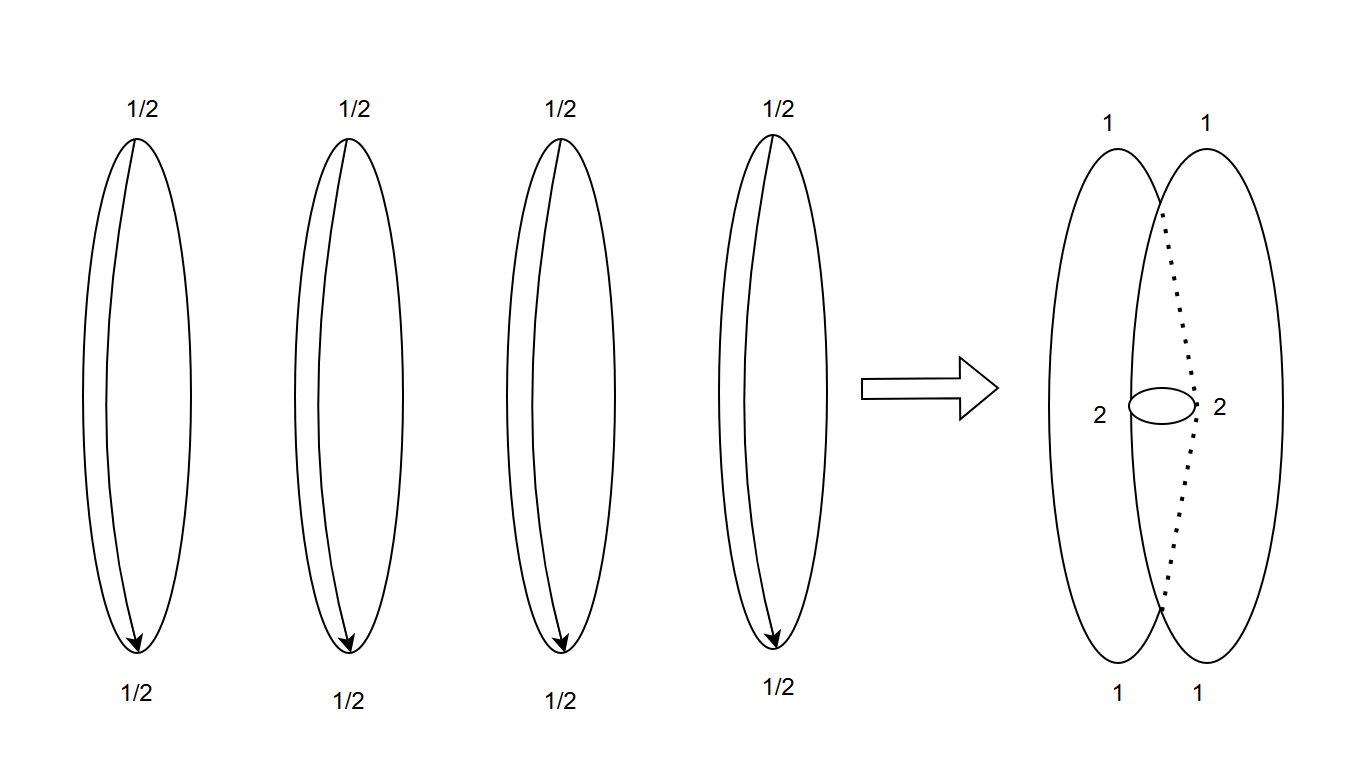}
\caption{Construction of $f(z) = \frac{(z - a_1)(z - a_2)}{(z - b_1)(z - b_2)}$.}
\label{fig:Method4}
\end{figure}

This construction is equivalent to the following procedure, shown in Figure~\ref{fig:Method5}: take two standard footballs, cut each along its equator, and glue the resulting pieces along these boundaries. The resulting surface corresponds to a rational map $f$ with two branch points and branch datum $\{[2],[2]\}$.

\begin{figure}[htbp]
\centering
\includegraphics[width=8cm]{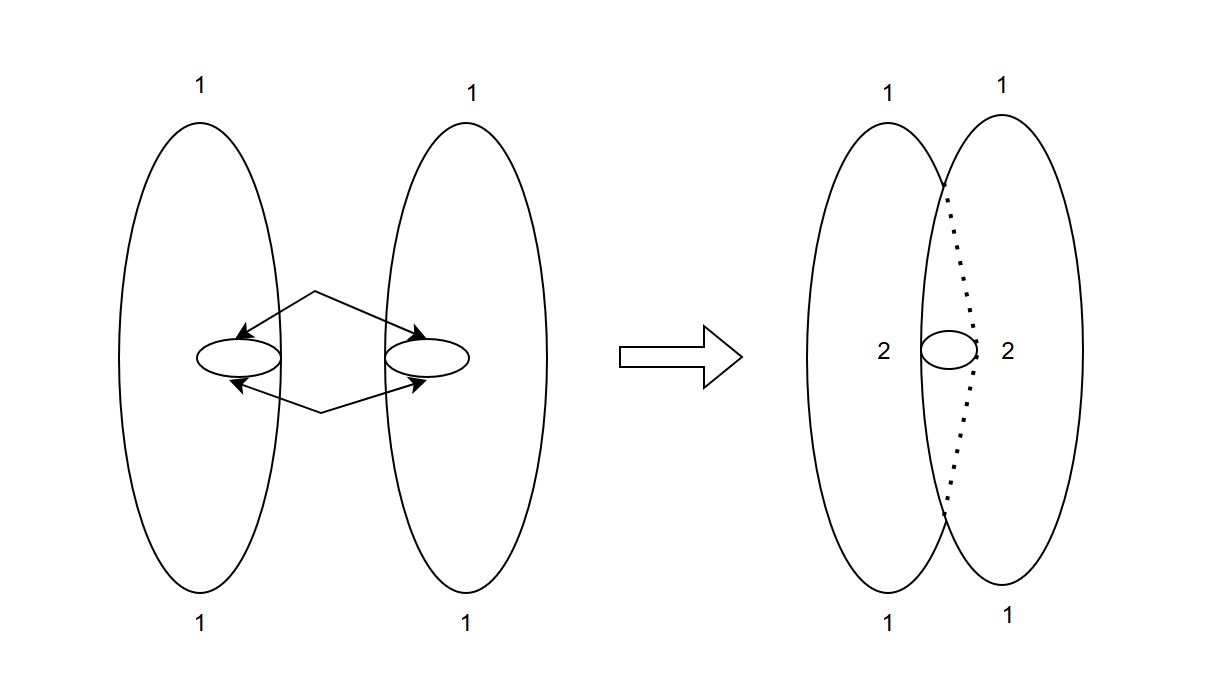}
\caption{Alternative construction of $f(z) = \frac{(z - a_1)(z - a_2)}{(z - b_1)(z - b_2)}$.}
\label{fig:Method5}
\end{figure}

\subsection{Some explicit examples}

In this section, we employ the first two elementary constructions from the previous section to illustrate the proof technique for Theorem~\ref{Thm2}. Example~\ref{Ex1} is the simplest case, as one partition in its branch datum has length~1. In contrast, Example~\ref{Ex4} is the most involved, requiring a more refined analysis. A thorough understanding of Example~\ref{Ex4} is essential to the proof of Theorem~\ref{Thm2}.

\begin{ex}\label{Ex1}
The branch datum $\mathcal{D} = \{[5], [2,2,1], [2,2,1]\}$ is realizable.
\end{ex}

\begin{proof}
First, one may readily verify that $\mathcal{D}$ is a candidate branch datum. Second, the minimum length among the partitions in $\mathcal{D}$ is $1$, and $\mathcal{D}$ contains $3$ partitions, satisfying $3 >1 + 1$.

Suppose there exists a rational map $f:\overline{\mathbb{C}}\rightarrow \overline{\mathbb{C}}$ with branch datum $\mathcal{D}$. Choose the point corresponding to the partition $[5]$ as the zero of $f$; that is, we can suppose that $f$ admits the expression given in \eqref{Exp_R}. Since it consists of a single number $5$, we examine the following candidate branch datum:
\[
\mathcal{D}_1 = \{[4], [2,1,1], [2,2]\}.
\]
Since $[2,2]$ has two elements equal to $2$, while $[2,1,1]$ has only one element greater than $1$, we proceed to the next candidate:
\[
\mathcal{D}_2 = \{[3], [2,1], [2,1]\}.
\]
Similarly, we consider:
\[
\mathcal{D}_3 = \{[2], [2]\}.
\]
We now construct the rational map $f$ with branch datum $\mathcal{D}$ by recursively constructing maps for $\mathcal{D}_3$, $\mathcal{D}_2$, and $\mathcal{D}_1$.

\noindent \textbf{Step 1}: Take two standard footballs and use \textbf{Method~1} to construct a rational map with branch datum $\mathcal{D}_3 = \{[2], [2]\}$, as illustrated in Figure~\ref{fig:Ex1}.

\noindent \textbf{Step 2}: Take one standard football $S^{2}_{\{1,1\}}$. Similarly to \textbf{Method~1}, attach it to the branched covering from Step~1(see Figure~\ref{fig:Ex1}). This yields a rational map with branch datum $\mathcal{D}_2 = \{[3], [2,1], [2,1]\}$.

\noindent \textbf{Step 3}: Take one standard football $S^{2}_{\{1,1\}}$ and attach it to the branched covering from Step~2 (see Figure~\ref{fig:Ex1}). This produces a rational map with branch datum $\mathcal{D}_1 = \{[4], [2,1^2], [2,2]\}$.

\noindent \textbf{Step 4}: Take one standard football $S^{2}_{\{1,1\}}$ and attach it to the branched covering from Step~3 (see Figure~\ref{fig:Ex1}). The result is a rational map with the desired branch datum
\[
\mathcal{D} = \{[5], [2,2,1], [2,2,1]\}.
\]

\begin{figure}[htbp]
\centering
\includegraphics[width=9cm]{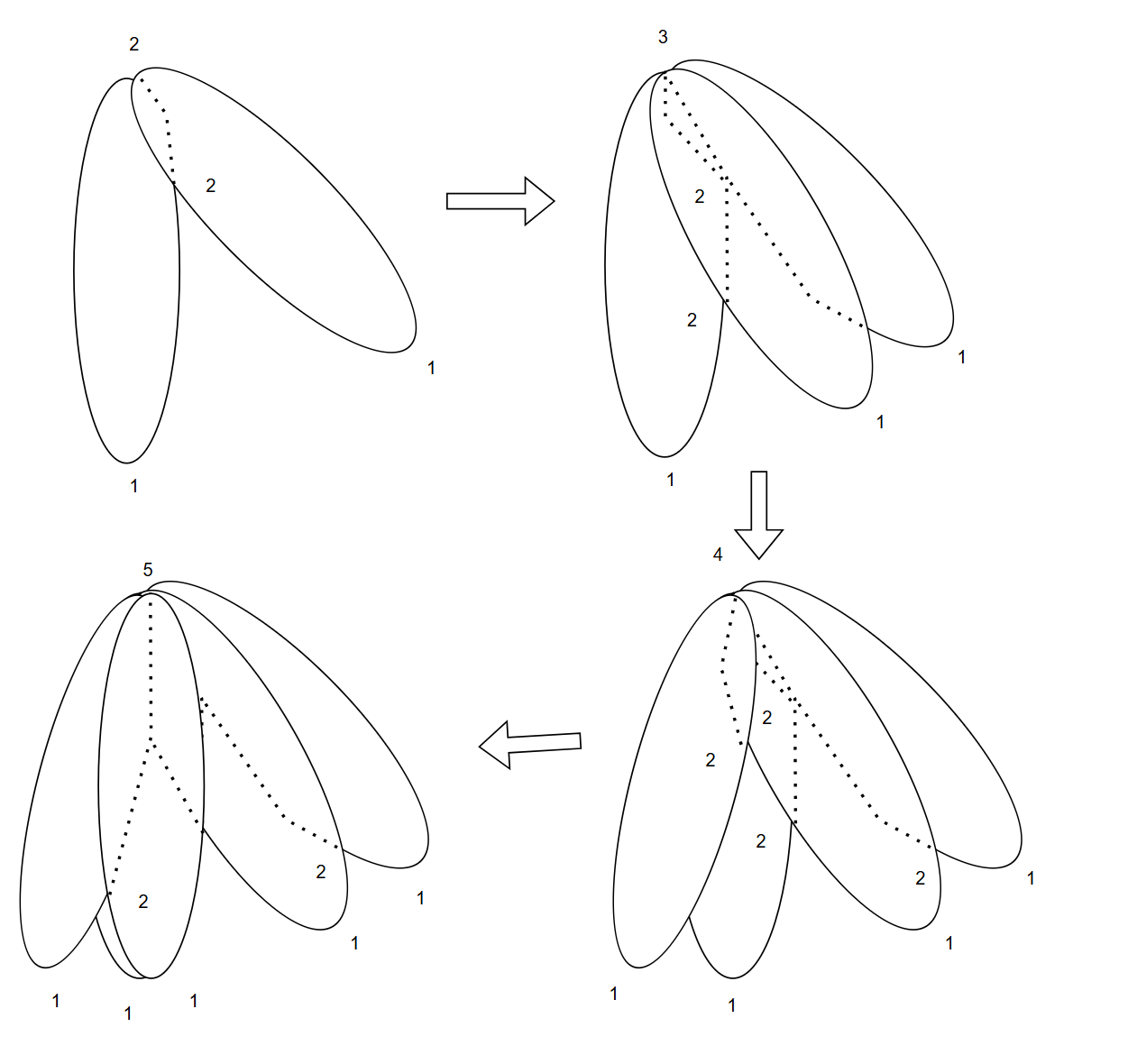}
\caption{Construction of $\mathcal{D} = \{[5], [2,2,1], [2,2,1]\}$.}
\label{fig:Ex1}
\end{figure}
\end{proof}

From Example~\ref{Ex1}, we observe that the rational function corresponding to the branch datum $\{[5], [2,2,1], [2,2,1]\}$ is not uniquely determined. Indeed, different constructions arise by rotating certain footballs.

\begin{ex}\label{Ex3}
The branch datum $\mathcal{D} = \{[4,1,1], [3,2,1], [2,2,1^2], [2,1^4], [2,1^4]\}$ is realizable.
\end{ex}

\begin{proof}
First, $\mathcal{D}$ is a candidate branch datum. Second, the minimum length is $3$, and $\mathcal{D}$ contains $5$ partitions, satisfying $5 >3 + 1$.

Suppose there exists a rational map $f:\overline{\mathbb{C}}\rightarrow \overline{\mathbb{C}}$ with branch datum $\mathcal{D}$. Choose the points corresponding to $[4,1,1]$ as the zeros. Since it contains $1$, we examine:
\[
\mathcal{D}_1 = \{[4,1], [3,1^2], [2,1^3], [2,1^3], [2,1^3]\}.
\]
As $[4,1]$ contains $1$, we proceed to:
\[
\mathcal{D}_2 = \{[4], [2,1^2], [2,1^2], [2,1^2]\}.
\]
Since the entry of $[4]$ exceeds $1$, we consider:
\[
\mathcal{D}_3 = \{[3], [2,1], [2,1]\}.
\]
Next, we consider:
\[
\mathcal{D}_4 = \{[2], [2]\}.
\]
We now recursively construct maps for $\mathcal{D}_4$, $\mathcal{D}_3$, $\mathcal{D}_2$, and $\mathcal{D}_1$.

\noindent \textbf{Step 1}: Take two standard footballs $S^{2}_{\{1,1\}}$. Using \textbf{Method~1}, construct a rational map with branch datum $\{[2], [2]\}$(see Figure~\ref{fig:Ex3}).

\noindent \textbf{Step 2}: Take one standard football $S^{2}_{\{1,1\}}$ and attach it to the branched covering from Step~1(see Figure~\ref{fig:Ex3}). This yields a rational map with branch datum $\mathcal{D}_3 = \{[3], [2,1], [2,1]\}$.

\noindent \textbf{Step 3}: Take one standard football $S^{2}_{\{1,1\}}$ and attach it to the branched covering from Step~2(see Figure~\ref{fig:Ex3}). This produces a rational map with branch datum $\mathcal{D}_2 = \{[4], [2,1^2], [2,1^2], [2,1^2]\}$.

\noindent \textbf{Step 4}: Take one standard football $S^{2}_{\{1,1\}}$ and attach it to the branched covering from Step~3(see Figure~\ref{fig:Ex3}). This gives a rational map with branch datum $\mathcal{D}_1 = \{[4,1], [3,1^2], [2,1^3], [2,1^3], [2,1^3]\}$.

\noindent \textbf{Step 5}: Take one standard football $S^{2}_{\{1,1\}}$ and attach it to the branched covering from Step~4 (see Figure~\ref{fig:Ex3}). The result is a rational map with the desired branch datum
\[
\mathcal{D} = \{[4,1,1], [3,2,1], [2,2,1^2], [2,1^4], [2,1^4]\}.
\]

\begin{figure}[htbp]
\centering
\includegraphics[width=11cm]{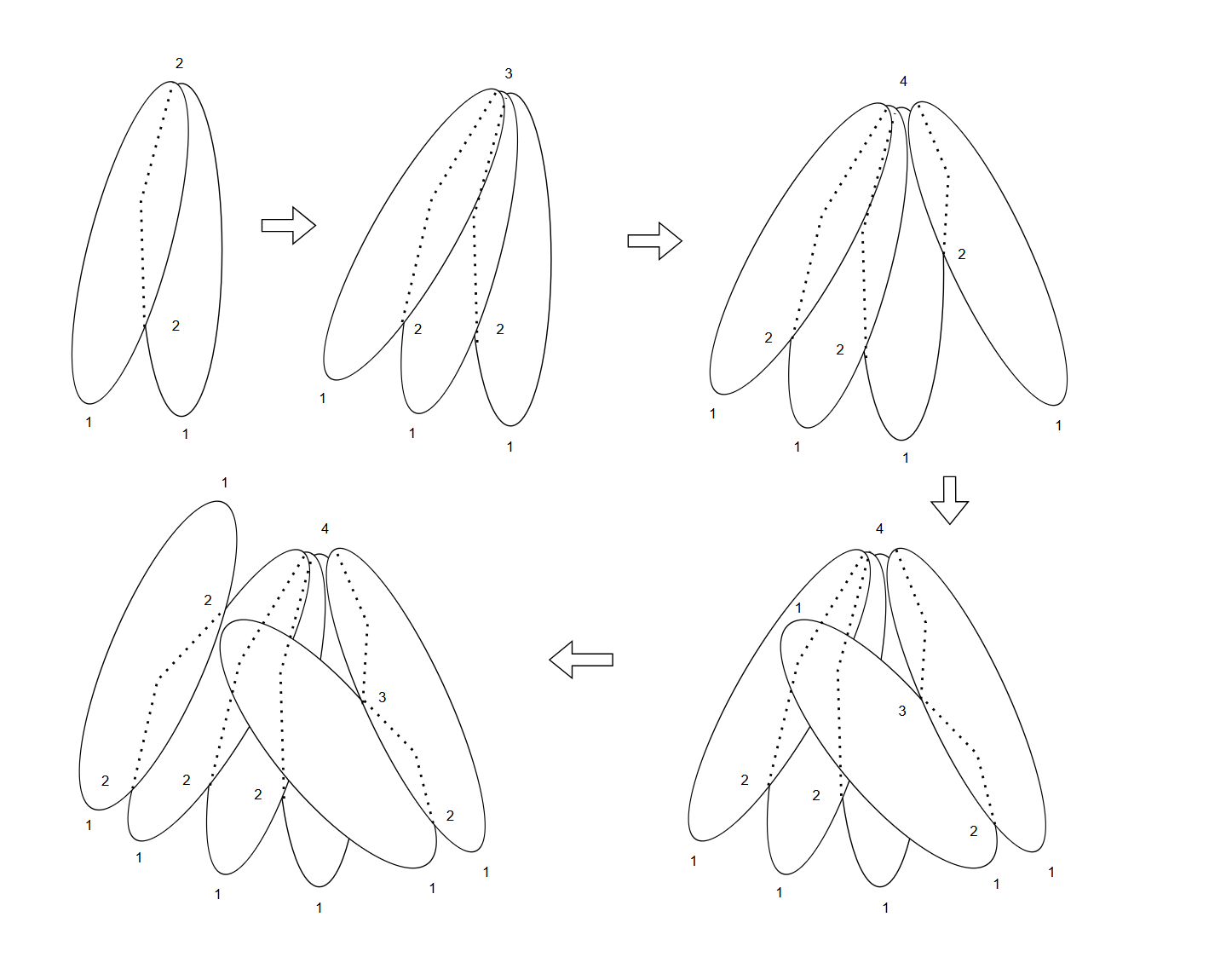}
\caption{Construction of $\mathcal{D} = \{[4,1,1], [3,2,1], [2,2,1^2], [2,1^4], [2,1^4]\}$.}
\label{fig:Ex3}
\end{figure}
\end{proof}

\begin{remark}
One may also choose the partition $[3,2,1]$ as the orders of the zeros of the desired rational map in Example \ref{Ex3}.
\end{remark}

\begin{ex}\label{Ex4}
The branch datum $\mathcal{D} = \{[4,2,2], [4,2,2], [2,2,1^4], [2,1^6], [2,1^6]\}$ is realizable.
\end{ex}

\begin{proof}
First, $\mathcal{D}$ is a candidate branch datum. Second, the minimum length is $3$, and $\mathcal{D}$ contains $5$ partitions, satisfying $5 > 3 + 1$.

Suppose there exists a rational map $f:\overline{\mathbb{C}}\rightarrow \overline{\mathbb{C}}$ with branch datum $\mathcal{D}$. Choose the points corresponding to $[4,2,2]$ as the zeros. Since all entries of $[4,2,2]$ are greater than $1$, we examine:
\[
\mathcal{D}_1 = \{[3,2,2], [4,2,1], [2,2,1^3], [2,1^5], [2,1^5]\}.
\]
We proceed to the next candidate:
\[
\mathcal{D}_2 = \{[2,2,2], [4,1^2], [2,2,1^2], [2,1^4], [2,1^4]\}.
\]
Next, we consider:
\[
\mathcal{D}_3 = \{[2,2,1], [4,1], [2,1^3], [2,1^3], [2,1^3]\}.
\]
Since $[2,2,1]$ contains $1$, we consider:
\[
\mathcal{D}_4 = \{[2,2], [3,1], [2,1,1], [2,1,1]\}.
\]
We then consider:
\[
\mathcal{D}_5 = \{[2,1], [2,1], [2,1], [2,1]\}.
\]
Finally, we consider:
\[
\mathcal{D}_6 = \{[2], [2]\}.
\]
We now recursively construct maps for $\mathcal{D}_l$, $l = 6, \ldots, 1$.

\noindent \textbf{Step 1}: Take two standard footballs $S^{2}_{\{1,1\}}$. Using \textbf{Method~1}, construct a rational map with branch datum $\mathcal{D}_6 = \{[2], [2]\}$ (see Figure~\ref{fig:Ex4}).

\noindent \textbf{Step 2}: Take one standard football $S^{2}_{\{1,1\}}$ and attach it to the branched covering from Step~1 (see Figure~\ref{fig:Ex4}). This yields a rational map with branch datum $\mathcal{D}_5 = \{[2,1], [2,1], [2,1], [2,1]\}$.

\noindent \textbf{Step 3}: Take one standard football $S^{2}_{\{1,1\}}$ and attach it to the branched covering from Step~2(see Figure~\ref{fig:Ex4}). This produces a rational map with branch datum $\mathcal{D}_4 = \{[2,2], [3,1], [2,1^2], [2,1^2]\}$.

\noindent \textbf{Step 4}: Take one standard football $S^{2}_{\{1,1\}}$ and attach it to the branched covering from Step~3(see Figure~\ref{fig:Ex4}). This gives a rational map with branch datum $\mathcal{D}_3 = \{[2,2,1], [4,1], [2,1^3], [2,1^3], [2,1^3]\}$.

\noindent \textbf{Step 5}: Take one standard football $S^{2}_{\{1,1\}}$ and attach it to the branched covering from Step~4(see Figure~\ref{fig:Ex4}). This yields a rational map with branch datum
\[
\mathcal{D}_2 = \{[2,2,2], [4,1,1], [2,2,1^2], [2,1^4], [2,1^4]\}.
\]

\noindent \textbf{Step 6}: Take one standard football $S^{2}_{\{1,1\}}$ and attach it to the branched covering from Step~5(see Figure~\ref{fig:Ex4}). This yields a rational map with branch datum
\[
\mathcal{D}_1 = \{[3,2,2], [4,2,1], [2,2,1^3], [2,1^5], [2,1^5]\}.
\]

\noindent \textbf{Step 7}: Take one standard football $S^{2}_{\{1,1\}}$ and attach it to the branched covering from Step~6 (as shown in Figure~\ref{fig:Ex4}). The result is a rational map with the desired branch datum
\[
\mathcal{D} = \{[4,2,2], [4,2,2], [2,2,1^4], [2,1^6], [2,1^6]\}.
\]
\begin{remark}
Another method to prove that the candidate branched datum $\mathcal{D}$ in Example \ref{Ex4} is realizable is as follows. Since the greatest common divisor of the two partitions $[4,2,2]$ and $[4,2,2]$ is $2$, and one can easily prove that both of the following two candidate branched data
$$
\{[2,2],[2,1^{2}],[2,1^{2}],[2,1^{2}],[2,1^{2}]\}
$$
and
$$
\{[2,1^{2}],[2,1^{2}],[2,1^{2}],[2,1^{2}],[2,1^{2}]\}
$$
are realizable, it follows that $\mathcal{D}$ is realizable \cite{SX20}.
\end{remark}

\begin{figure}[htbp]
\centering
\includegraphics[width=12cm]{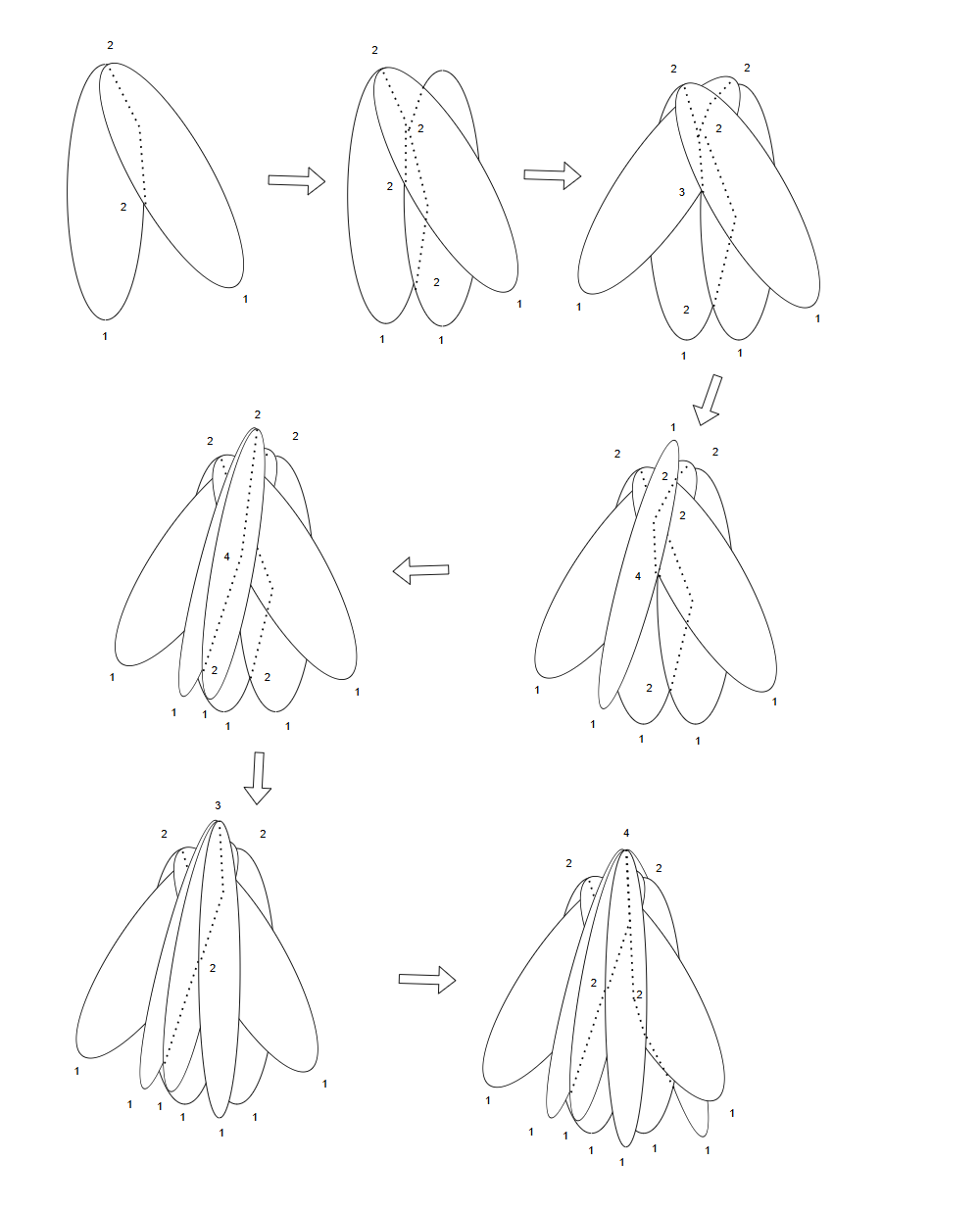}
\caption{Construction of $\mathcal{D} = \{[4,2,2], [4,2,2], [2,2,1^4], [2,1^6], [2,1^6]\}$.}
\label{fig:Ex4}
\end{figure}
\end{proof}

\section{Proof of Theorem \ref{Thm2}}

To prove Theorem~\ref{Thm2}, we first establish some preliminary properties about the partitions.

Let
\begin{equation}\label{notation}
\pi_i =\big[ a_{i1}, \dots, a_{ir_i}\big]= \big[ b_{i1}, \dots, b_{iq_i}, \underbrace{1, \dots, 1}_{e_i} \big], \quad i = 1, 2, \dots, k,
\end{equation}
be \( k \) nontrivial partitions of an integer \( d \geq 3 \), where for each \( i \), the entries satisfy \( b_{i1} \geq b_{i2} \geq \dots \geq b_{iq_i} \geq 2 \). Assume the collection
\[
\{ \pi_1, \dots, \pi_k \}
\]
satisfies the Riemann--Hurwitz formula and that \( |\pi_k| = q_k + e_k \leq |\pi_i| = q_i + e_i \) for all \( i = 1, \dots, k-1 \).

We then have the following:

\begin{lem}\label{le8.1}
Let \( k \geq q_k + e_k + 2 \). If \( q_1 \geq q_i \) for all \( i = 2, \dots, k-1 \), then \( e_i > 0 \) for all \( i = 2, \dots, k-1 \), and consequently \( q_i + e_i > \frac{d}{2} \).
\end{lem}

\begin{proof}
Assume, for contradiction, that \( e_t = 0 \) for some \( t \in \{2, \dots, k-1\} \). Then \( \pi_t = [b_{t1}, \dots, b_{tq_t}] \). By the Riemann--Hurwitz formula:
\[
\begin{aligned}
2d - 2 &= \sum_{i=1}^{k} \sum_{j=1}^{q_i} (b_{ij} - 1) \\
&= \big[d - (q_k + e_k)\big] + (d - q_t) + \sum_{\substack{i=1 \\ i \ne t}}^{k-1} \sum_{j=1}^{q_i} (b_{ij} - 1).
\end{aligned}
\]
Rearranging terms, we obtain:
\[
\begin{aligned}
q_t &= 2 + \sum_{\substack{i=1 \\ i \ne t}}^{k-1} \sum_{j=1}^{q_i} (b_{ij} - 1) - (q_k + e_k) \\
&\geq 2 + \sum_{\substack{i=1 \\ i \ne t}}^{k-1} q_i - (q_k + e_k) \quad (\text{since } b_{ij} \geq 2) \\
&= 2 + q_1 + \sum_{\substack{i=2 \\ i \ne t}}^{k-1} q_i - (q_k + e_k) \\
&\geq 2 + q_1 + (k - 3) - (q_k + e_k) \quad (\text{since } q_i \geq 1) \\
&= q_1 + k - (q_k + e_k + 1).
\end{aligned}
\]
Given \( k \geq q_k + e_k + 2 \), it follows that:
\[
q_t > q_1,
\]
contradicting the assumption \( q_1 \geq q_t \). Therefore, \( e_i > 0 \) for all \( i = 2, \dots, k-1 \). Moreover, since \( \pi_i \) is a partition of \( d \), we have \( q_i + e_i > \frac{d}{2} \) for all \( i = 2, \dots, k-1 \).
\end{proof}

\begin{lem}\label{le8.2}
Let \( k \geq q_k + e_k + 2 \). Suppose \( e_k = 0 \), and assume \( q_1 \geq q_i \) for all \( i = 2, \dots, k-1 \). If \( q_1 = q_i \) for some \( i \in \{2, \dots, k-1\} \), further assume \( b_{11} \geq b_{i1} \). Define
\[
\begin{aligned}
&\hat{\pi}_1 = \big[ b_{11}, \dots, b_{1q_1} - 1, \underbrace{1, \dots, 1}_{e_1} \big], \\
&\hat{\pi}_i = \big[ b_{i1}, \dots, b_{iq_i}, \underbrace{1, \dots, 1}_{e_i - 1} \big], \quad i = 2, \dots, k-1 \quad (\text{by Lemma~\ref{le8.1}, } e_i > 0), \\
&\hat{\pi}_k = \big[ b_{k1} - 1, \dots, b_{kq_k} \big].
\end{aligned}
\]
Then the following hold.

\noindent \textbf{Case 1:} \(\hat{\pi}_1 \neq [1, \dots, 1]\)
\begin{enumerate}
    \item \(\hat{\pi}_1, \dots, \hat{\pi}_k\) are \(k\) nontrivial partitions of \(d - 1\).
    \item The partitions \(\hat{\pi}_1, \dots, \hat{\pi}_k\) satisfy the Riemann--Hurwitz formula.
    \item \(q_k \leq q_i + e_i\) for all \(i = 1, \dots, k - 1\).
\end{enumerate}

\noindent \textbf{Case 2:} \(\hat{\pi}_1 = [1, \dots, 1]\)
\begin{enumerate}
    \item \(\pi_i = [2, \underbrace{1, \dots, 1}_{d - 2}]\) for all \(i = 1, \dots, k - 1\), and \(k = d + q_k + 1 > q_k + 2\).
    \item The partitions \(\hat{\pi}_2, \dots, \hat{\pi}_k\) satisfy the Riemann--Hurwitz formula.
    \item \(q_k \leq q_i + e_i\) for all \(i = 2, \dots, k - 1\).
\end{enumerate}
\end{lem}

\begin{proof}
If \( e_k = 0 \), then since \( \pi_k \) is a nontrivial partition of \( d \), we have \( q_k \leq \frac{d}{2} \).

\noindent\textbf{Case 1:} \(\hat{\pi}_1 \neq [1, \dots, 1]\). A direct computation verifies statements (1), (2), and (3).

\noindent\textbf{Case 2:} \(\hat{\pi}_1 = [1, \dots, 1]\). Then \( q_1 = 1 \) and \( b_{11} = 2 \). Moreover, by our assumption,
\[
\pi_i = [2, \underbrace{1, \dots, 1}_{d - 2}], \quad \text{for all } i = 1, \dots, k - 1.
\]
By the Riemann--Hurwitz formula,
\[
2d - 2 = (k - 1) + (d - q_k),
\]
which implies
\[
k = d + q_k + 1 > q_k + 2.
\]
Again, a direct calculation confirms statements (1), (2), and (3).
\end{proof}

\begin{lem}\label{le8.3}
Let \( k \geq q_k + e_k + 2 \). Suppose \( e_k > 0 \), and assume \( q_1 \geq q_2 \geq q_i \) for all \( i = 3, \dots, k-1 \). If \( q_1 = q_i \) or \( q_2 = q_i \) for some \( i \in \{2, \dots, k-1\} \), further assume \( b_{11} \geq b_{i1} \). Define
\[
\begin{aligned}
&\hat{\pi}_j = \big[ b_{j1}, \dots, b_{jq_j} - 1, \underbrace{1, \dots, 1}_{e_j} \big], \quad j = 1, 2, \\
&\hat{\pi}_i = \big[ b_{i1}, \dots, b_{iq_i}, \underbrace{1, \dots, 1}_{e_i - 1} \big], \quad i = 3, \dots, k-1 \quad (\text{by Lemma~\ref{le8.1}, } e_i > 0), \\
&\hat{\pi}_k = \big[ b_{k1}, \dots, b_{kq_k}, \underbrace{1, \dots, 1}_{e_k - 1} \big].
\end{aligned}
\]
Then the following hold.

\noindent \textbf{Case 1:} \(\hat{\pi}_2 \neq [1, \dots, 1]\)
\begin{enumerate}
    \item \(\hat{\pi}_1, \dots, \hat{\pi}_k\) are \(k\) nontrivial partitions of \(d - 1\).
    \item The partitions \(\hat{\pi}_1, \dots, \hat{\pi}_k\) satisfy the Riemann--Hurwitz formula.
    \item \(q_k + e_k - 1 \leq q_1 + e_1\) and \(q_k + e_k - 1 \leq q_i + e_i - 1\) for all \(i = 2, \dots, k - 1\).
\end{enumerate}

\noindent \textbf{Case 2:} \(\hat{\pi}_1 \neq [1, \dots, 1]\) and \(\hat{\pi}_2 = [1, \dots, 1]\)
\begin{enumerate}
    \item \(\pi_i = [2, \underbrace{1, \dots, 1}_{d - 2}]\) for all \(i = 2, \dots, k - 1\), and \(k = q_k + e_k + q_1 + e_1 \geq q_k + e_k + 2\).
    \item The partitions \(\hat{\pi}_1, \hat{\pi}_3, \dots, \hat{\pi}_k\) satisfy the Riemann--Hurwitz formula.
    \item \(q_k + e_k - 1 \leq q_1 + e_1\) and \(q_k + e_k - 1 \leq q_i + e_i\) for all \(i = 3, \dots, k - 1\).
\end{enumerate}

\noindent \textbf{Case 3:} \(\hat{\pi}_1 = [1, \dots, 1]\)
\begin{enumerate}
    \item \(\pi_i = [2, \underbrace{1, \dots, 1}_{d - 2}]\) for all \(i = 1, \dots, k - 1\), and \(k = d + q_k + e_k - 1 \geq q_k + e_k + 2\).
    \item The partitions \(\hat{\pi}_2, \dots, \hat{\pi}_k\) satisfy the Riemann--Hurwitz formula.
    \item \(q_k + e_k - 1 \leq q_i + e_i - 1\) for all \(i = 2, \dots, k - 1\).
\end{enumerate}
\end{lem}

\begin{proof}
The assertions follow by direct computation.
\end{proof}

The proof of Theorem~\ref{Thm2} proceeds by induction on the degree $d$. The core idea is geometric and constructive: we build the desired rational map by inductively gluing standard footballs to a branched covering of lower degree. The induction hypothesis provides a "seed" map whose own geometric decomposition is known, and the inductive step amounts to surgically attaching one more football to this structure, thereby increasing the degree by one and precisely achieving the required branching datum. Clearly, this implies the following result, which was first proved by Bara\'{n}ski for the case of rational maps (see Proposition 11 in \cite{Bar01}, though it contains a typo). Furthermore, this result also extends to meromorphic functions.

\begin{prop}
Under the notation of (\ref{notation}), suppose the candidate branch datum \(\mathcal{D} = \{\pi_1, \ldots, \pi_k\}\) of degree \(d\) is realizable. Then for any indices \(1 \leq i_0 \leq k\) and \(1 \leq j_0 \leq r_{i_0}\), the candidate branch datum \(\hat{\mathcal{D}} = \{\hat{\pi}_1, \ldots, \hat{\pi}_k, \hat{\pi}_{k+1}\}\) of degree \(d+1\) defined by
\[
\hat{\pi}_i = [a_{i1}, \ldots, a_{i r_i}, 1], \quad i \neq i_0, k+1,
\]
\[
\hat{\pi}_{i_0} = [a_{i_0,1}, a_{i_0,2}, \ldots, a_{i_0, j_0} + 1, a_{i_0, j_0+1}, \ldots, a_{i_0, r_{i_0}}],
\]
and
\[
\hat{\pi}_{k+1} = [2, \underbrace{1, \ldots, 1}_{d-1}],
\]
is also realizable.

Similarly, if \(\mathcal{D}\) is realizable, then the candidate branch datum \(\hat{\mathcal{D}} = \{\hat{\pi}_1, \ldots, \hat{\pi}_k, \hat{\pi}_{k+1}, \hat{\pi}_{k+2}\}\) of degree \(d+1\) defined by
\[
\hat{\pi}_i = [a_{i1}, \ldots, a_{i r_i}, 1], \quad 1 \leq i \leq k,
\]
and
\[
\hat{\pi}_{k+1} = \hat{\pi}_{k+2} = [2, \underbrace{1, \ldots, 1}_{d-1}],
\]
is realizable as well.
\end{prop}

 We now proceed to the proof of Theorem~\ref{Thm2}.

For the base case \( d = 3 \), we have either:
\begin{itemize}
    \item \( k = 3 \) with \( \pi_1 = \pi_2 = [2, 1] \), \( \pi_3 = [3] \), or
    \item \( k = 4 \) with \( \pi_i = [2, 1] \) for all \( i = 1, 2, 3, 4 \).
\end{itemize}

In the case \( k = 3 \) with branch data \( \{[2,1], [2,1], [3]\} \), a rational map \( f : \overline{\mathbb{C}} \to \overline{\mathbb{C}} \) can be constructed by gluing three standard footballs \( S^2_{\{1,1\}} \), as illustrated in Figure~\ref{fig:Proof1}. The figure shows that the vertex corresponding to the partition \( [3] \) is connected to the vertex for the entry \( 2 \) in \( [2,1] \) via a standard football.

\begin{figure}[htbp]
\centering
\includegraphics[width=7cm]{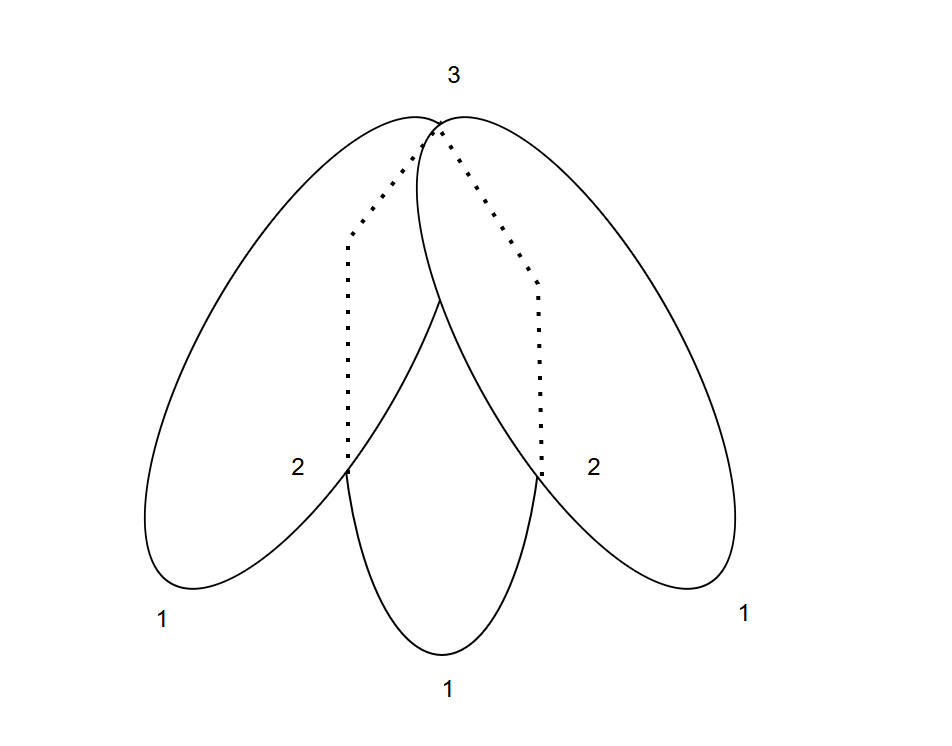}
\caption{Construction for branch data \( \{[2,1], [2,1], [3]\} \).}
\label{fig:Proof1}
\end{figure}

For the case \( k = 4 \) with all \( \pi_i = [2,1], i=1,2,3,4 \), a rational map can be similarly constructed by gluing standard footballs, as shown in Figure~\ref{fig:Proof2}. Here, the vertex for the entry \( 2 \) in one partition \( [2,1] \) is connected to the vertex for the entry \( 2 \) in another \( [2,1] \) via a standard football.

\begin{figure}[htbp]
\centering
\includegraphics[width=7cm]{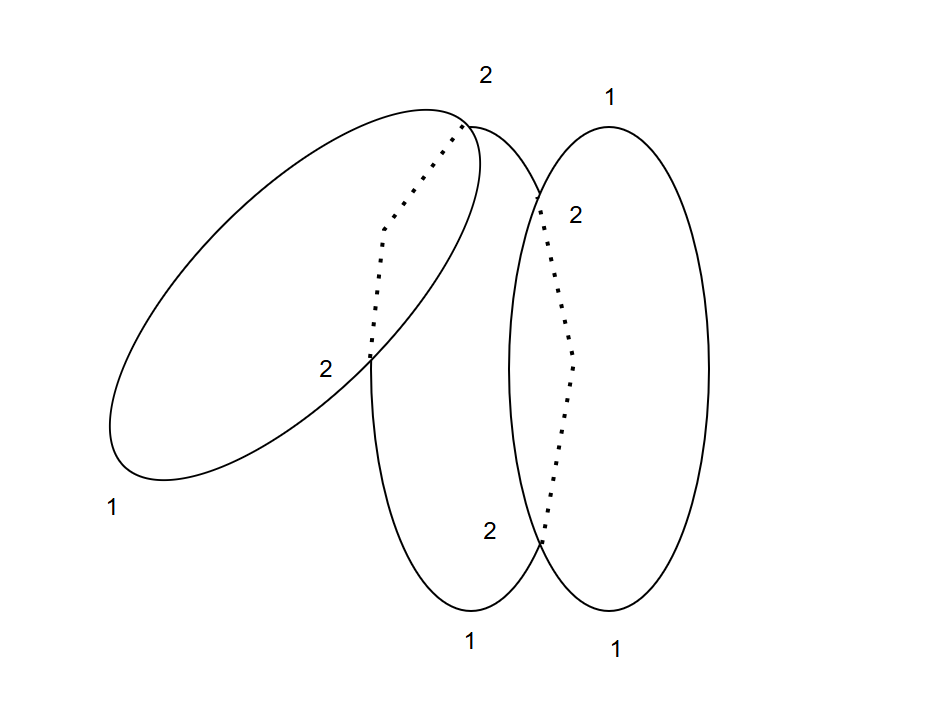}
\caption{Construction for branch data \( \{[2,1], [2,1], [2,1], [2,1]\} \).}
\label{fig:Proof2}
\end{figure}

We now proceed by induction. Assume the result holds for \( d = n - 1 \), and that in the corresponding rational map:
\begin{itemize}
    \item If all entries of \( \pi_k \) are greater than 1, then the vertex for \( b_{k,1} \) in \( \pi_k \) is connected to the vertex for \( b_{1q_1} \) in \( \pi_1 \) by a standard football.
    \item If \( \pi_k \) contains the entry 1, then the vertex for \( b_{1q_1} \) in \( \pi_1 \) is connected to the vertex for \( b_{2q_2} \) in \( \pi_2 \) by a standard football.
\end{itemize}

For \( d = n \), we construct a rational map \( f \) of degree \( n \) with branch data \( \{ \pi_1, \dots, \pi_k \} \) as follows.

\noindent \textbf{Case 1}: All entries of \( \pi_k \) are greater than 1.
Define
\[
\begin{aligned}
\hat{\pi}_1 &= \big[ b_{11}, \dots, b_{1q_1} - 1, \underbrace{1, \dots, 1}_{e_1} \big], \\
\hat{\pi}_i &= \big[ b_{i1}, \dots, b_{iq_i}, \underbrace{1, \dots, 1}_{e_i - 1} \big], \quad i = 2, \dots, k-1, \\
\hat{\pi}_k &= \big[ b_{k1}-1, \dots, b_{kq_k} \big].
\end{aligned}
\]

\noindent \textbf{Case 2}: \( \pi_k \) contains the entry 1.
Define
\[
\begin{aligned}
\hat{\pi}_1 &= \big[ b_{11}, \dots, b_{1q_1} - 1, \underbrace{1, \dots, 1}_{e_1} \big], \\
\hat{\pi}_2 &= \big[ b_{21}, \dots, b_{2q_2} - 1, \underbrace{1, \dots, 1}_{e_2} \big], \\
\hat{\pi}_i &= \big[ b_{i1}, \dots, b_{iq_i}, \underbrace{1, \dots, 1}_{e_i - 1} \big], \quad i = 3, \dots, k.
\end{aligned}
\]

In both cases, the partitions \( \hat{\pi}_1, \dots, \hat{\pi}_k \) are nontrivial partitions of \( n - 1 \) satisfying the Riemann--Hurwitz formula. By the induction hypothesis, there exists a rational map \( g \) with this branch datum such that:
\begin{itemize}
    \item In Case 1, the vertex for \( b_{k1} - 1 \) in \( \hat{\pi}_k \) is connected to the vertex for \( b_{1q_1} - 1 \) in \( \hat{\pi}_1 \) by a standard football.
    \item In Case 2, the vertex for \( b_{1q_1} - 1 \) in \( \hat{\pi}_1 \) is connected to the vertex for \( b_{2q_2} - 1 \) in \( \hat{\pi}_2 \) by a standard football.
\end{itemize}

Attaching an additional standard football along the corresponding geodesic segments in the domain of \( g \) yields the desired rational map \( f \) of degree \( n \), completing the induction step.

\section{Concluding remarks and open problems}

The geometric decomposition established in Theorem~\ref{Thm1} provides a new
perspective on rational maps and their branching structures. Several
directions for future investigation emerge naturally:

\begin{enumerate}
\item \textbf{Sharpness of the bound:} Theorem~\ref{Thm2} gives a sufficient
      condition $k > l + 1$ for realizability under a more condition. Is this bound also
      necessary under suitable hypotheses, i.e., whether Conjecture \ref{conj} holds in general? The non-realizable example
      in Theorem~1.3 suggests that the bound is sharp in certain cases,
      but a complete characterization remains open.

\item \textbf{Beyond rational maps:} The football decomposition may
      extend to branched coverings of higher-genus surfaces, potentially
      shedding new light on the Hurwitz existence problem for arbitrary
      sources.

\item \textbf{Moduli aspects:} The construction yields families of
      rational maps with prescribed branch datum. Understanding the moduli
      of such maps through the lens of football gluings remains an
      interesting open question.
\end{enumerate}
\textbf{Statements and Declarations}

\textbf{Data Availability Statement}  This manuscript has no associated data.

\textbf{Competing interests} On behalf of all authors, the corresponding author states that there is no conflict of interest.

\section*{Acknowledgments}
~~~~
We would like to thank Professor Xu Bin for valuable suggestions.

%
%
%
%

\end{document}